\documentclass{amsart}
\usepackage{amssymb} 
\usepackage{amsmath,amsthm}
\usepackage{graphicx}

\usepackage{subfig}

\theoremstyle{plain}
\newtheorem{thm}{Theorem}[section]
\newtheorem{lem}[thm]{Lemma}

\newtheorem{cor}[thm]{Corollary}

\numberwithin{equation}{section}

\begin{document}

\title[Helicoidal surfaces rotating/translating under the MCF]{Helicoidal surfaces rotating/translating under the mean curvature flow}

\author{Hoeskuldur P. Halldorsson}
\address{MIT, Department of Mathematics, 77 Massachusetts Avenue, Cambridge, MA 02139-4307.}
\email{hph@math.mit.edu}


\date{} 


\begin{abstract}
We describe all possible self-similar motions of immersed hypersurfaces in Euclidean space under the mean curvature flow and derive the corresponding hypersurface equations. Then we present a new two-parameter family of immersed helicoidal surfaces that rotate/translate with constant velocity under the flow. We look at their limiting behaviour as the pitch of the helicoidal motion goes to $0$ and compare it with the limiting behaviour of the classical helicoidal minimal surfaces. Finally, we give a classification of the immersed cylinders in the family of constant mean curvature helicoidal surfaces.
\end{abstract}

\maketitle

\section{Introduction}
The mean curvature flow (MCF) of immersed hypersurfaces in Euclidean space is defined as follows.
Let $M^n$ be an $n$-dimensional manifold and consider a family of smooth immersions $F_t = F(\cdot,t): M^n \rightarrow \mathbf R^{n+1}$, $t\in I$, with $M_t = F_t(M^n)$. The family of hypersurfaces $(M_t)_{t\in I}$ is said to \emph{move by mean curvature} if
\begin{equation*}
\frac{\partial F}{\partial t}(p,t) = \mathbf H(F(p,t))
\end{equation*} 
for $p \in M^n$ and $t \in I$. Here $\mathbf H$ is the mean curvature vector of $M_t$, given by $\mathbf H = -H\mathbf n$, where $\mathbf n$ is a choice of unit normal field and $H = \text{div}_{M_t}\mathbf n$ is the mean curvature of $M_t$.

Hypersurfaces which move in a self-similar manner under the MCF play an important role in the singularity theory of the flow. The most important ones are the shrinking hypersurfaces and the translating hypersurfaces, which model the so-called type 1 and type 2 singularities respectively (see \cite{huisksin}). Examples of shrinking hypersurfaces can for example be found in \cite{abrlang}, \cite{ang}, \cite{chopp}, \cite{kapkleenmoell} and  \cite{kleenmoell}, and examples of translating hypersurfaces in \cite{cluttschnurschulz}, \cite{nguyen} and \cite{wang}. Other types of self-similar motions under the flow have been studied less thoroughly. Some examples of expanding hypersurfaces can be found in \cite{angilmchopp} and \cite{eckhuisk}. In \cite{hph}, the author gave a complete classification of all self-similar solutions to the flow in the case $n=1$, the so-called curve shortening flow in the plane. However, the classification problem is significantly harder in higher dimensions.

In this paper, we present a new two-parameter family of complete immersed surfaces in $\mathbf R^3$ which rotate under the MCF. The surfaces belong to the family of so-called helicoidal surfaces, i.e., surfaces invariant under a helicoidal motion. Since the surfaces rotate around their helicoidal axis, they can also be viewed as translating solutions to the flow. However, they are not convex. The idea to look for these surfaces came from \cite{hungsmocz}.

The paper is structured as follows.
In Section \ref{selfsim}, we describe all possible self-similar motions of immersed hypersurfaces under the MCF and derive the equations the hypersurfaces have to satisfy.
It turns out that the only possible self-similar motions are the following:
\begin{itemize}
\item dilation with scaling function $\sqrt{2bt+1}$, i.e., shrinking with increasing speed if $b<0$ and expanding with decreasing speed if $b>0$,
\item translation with constant velocity,
\item rotation with constant angular velocity,
\item dilation and rotation together, with scaling function $\sqrt{2bt+1}$ and angular velocity in a fixed direction but with angular speed proportional to $\frac{\log(2bt+1)}{2b}$,
\item translation and rotation together but in orthogonal directions, with constant velocity and angular velocity.
\end{itemize}

In Section \ref{helicoidalsurfaces} we introduce helicoidal surfaces. In this paper, we focus on the planar curves given as the intersection of the surfaces with a plane orthogonal to their helicoidal axis. These curves generate the surfaces and give rise to a parametrization different from the one most common in the literature. Using this approach, we derive a new method for constructing helicoidal surfaces with prescribed mean curvature. This method is used in the sections that follow. 

Section \ref{rotatingsurfaces} contains the heart of the paper. There we present the two-parameter family of helicoidal surfaces that rotate with unit speed under the mean curvature flow. We derive some basic properties of these surfaces and investigate the limiting behaviour of the generating curves as the pitch of the helicoidal motion goes to 0. A specific choice of initial values gives convergence to a circle, so the corresponding surfaces in some sense converge to a cylinder.

In Section \ref{minimalsurfaces} we take a look at the classical helicoidal minimal surfaces described in \cite{wund}. We derive their parametrization and investigate their limiting behavior as the pitch goes to 0.  Finally, in Section \ref{cmcsurfaces} we look at the classical constant mean curvature helicoidal surfaces, studied in \cite{docarmdajcz}, \cite{hittrouss} and \cite{perd}. We derive their parametrization and in Theorem \ref{immersedcylinder}, we classify the immersed cylinders in this family of surfaces, in the spirit of the classification of the self-shrinking solutions to the curve shortening flow in the plane given by Abresch and Langer in \cite{abrlang}.

\section{Self-similar motions under the mean curvature flow}
\label{selfsim}

Let $\Sigma^n$ be a complete $n$-dimensional hypersurface immersed in $\mathbf R^{n+1}$. We call the immersion $F$ and to simplify notation we identify $\Sigma$ with $F(\Sigma)$. A self-similar motion of $\Sigma$ is a family $(\Sigma_t)_{t\in I}$  of immersed hypersurfaces where the immersions are of the form
\begin{equation*}
F(p,t) = g(t)Q(t)F(p) + \mathbf v(t), \quad p\in\Sigma,t\in I.
\end{equation*}
Here $I$ is an interval containing $0$ and $g:I \rightarrow \mathbf R$, $Q: I \rightarrow SO(n+1)$ and $\mathbf v:I\rightarrow \mathbf R^{n+1}$ are differentiable functions s.t. $g(0)=1$, $Q(0) = I$ and $\mathbf v(0) = \mathbf 0$, and hence $F(p,0) = F(p)$, i.e., $\Sigma_0 = \Sigma$.

This motion is the mean curvature flow of $\Sigma$ (up to tangential diffeomorphisms) if and only if the equation
\begin{equation*}
\left\langle \frac{\partial F}{\partial t}(p,t),\mathbf n(p,t) \right\rangle = -H(p,t)
\end{equation*}
holds for all $p\in\Sigma$, $t\in I$. Simple calculations yield that this equation is equivalent to
\begin{equation}
\label{FullJafna}
\begin{aligned}
 g(t)g'(t)\langle F(p),\mathbf n(p)\rangle &+ g^2(t)\langle Q^{T}(t)Q'(t)F(p),\mathbf n(p) \rangle \\
&+ g(t)\langle Q^{T}(t)\mathbf v'(t),\mathbf n(p)\rangle = -H(p).
\end{aligned}
\end{equation}
By looking at this equation at time $t=0$, we see that $\Sigma$ has to satisfy
\begin{equation}
\label{HofudJafna}
b\langle F,\mathbf n \rangle +\langle AF,\mathbf n \rangle +  \langle \mathbf c,\mathbf n \rangle = -H
\end{equation}
where $b = g'(0)$, $A=Q'(0)\in \mathfrak{so}(n+1)$ and $\mathbf c= \mathbf v'(0)$. It turns out that satisfying an equation of this form is also a sufficient condition for $\Sigma$ to move in a self-similar manner under the MCF. To see that, we look separately at two cases.\\
\\
{\bf Dilation and rotation:}
First let's look at the case where there is no translation term. Assume $\Sigma$ satisfies
\begin{equation}
\label{RotScal}
b\langle F,\mathbf n \rangle  + \langle AF,\mathbf n \rangle    = -H.
\end{equation}
Now, if the functions $g$ and $Q$ satisfy $g(t)g'(t) = b$ and $g^2(t)Q^{T}(t)Q'(t) = A$ for all $t\in I$, then Equation \eqref{FullJafna} is satisfied for all $p\in \Sigma$, $t\in I$. Solving these differential equations with our initial values gives
\begin{equation*}
g(t) = \sqrt{2bt+1} \quad \text{and} \quad Q(t) = \begin{cases} \exp(\frac{\log(2bt+1)}{2b}A) & \text{if } b\neq0,\\
\exp(tA) & \text{if } b=0. \end{cases} 
\end{equation*}
Therefore, under the MCF the hypersurface $\Sigma$ either expands and rotates forever with decreasing speed ($b >0$), rotates forever with constant speed ($b=0$) or shrinks and rotates with increasing speed until a singularity forms at time $t = -\frac{1}{2b}$ $(b<0)$. Of course, there is no rotation if $A=0$.  Note that when $b=-1$ and $A=0$, Equation \eqref{RotScal} reduces to the famous self-shrinker equation. In \cite{hph}, the author showed the existence of all these motions in the case of the curve shortening flow in the plane.\\
\\
{\bf Translation and rotation:}
Now let's look at the case where there is no dilation. Assume $\Sigma$ satisfies
\begin{equation}
\label{RotTran}
\langle AF,\mathbf n \rangle + \langle \mathbf c,\mathbf n \rangle = -H
\end{equation}
and furthermore impose the condition $A\mathbf c = 0$, i.e., the translation and rotation are in orthogonal directions.
If the functions $Q$ and $\mathbf v$ satisfy $Q^{T}(t)Q'(t) = A$ and $Q^{T}(t)\mathbf v'(t)=\mathbf c$ for all $t\in I$, then Equation \eqref{FullJafna} is satisfied for all $p \in \Sigma$, $t\in I$. Solving these differential equations with our initial values yields
\begin{equation*}
Q(t) = \exp(tA) \quad \text{and} \quad \mathbf v(t) = t\mathbf c. 
\end{equation*}
Therefore, under the MCF the hypersurface translates and rotates forever with constant velocity and angular velocity, and these motions are orthogonal.\\

But what happens if we include all three terms in the left hand side of Equation \eqref{HofudJafna}? Assume  $\Sigma$ satisfies \eqref{HofudJafna}. By translating $\Sigma$ by a vector $\mathbf w$, we get a hypersurface $\hat \Sigma$ satisfying the equation
\begin{equation*}
b\langle \hat F,\mathbf{\hat n} \rangle +\langle A\hat F,\mathbf{\hat n} \rangle +  \langle \mathbf c - (A+bI)\mathbf w,\mathbf {\hat n} \rangle = -\hat H.
\end{equation*}
Note that since $A \in \mathfrak{so}(n+1)$, $A$ is skew-symmetric, i.e., $A^{\mathsf T}=-A$. Therefore, $A$ has no non-zero real eigenvalues and $\ker(A) = (\text{im}(A))^{\perp}$. Thus, we can make the following choice of $\mathbf w$:
\begin{itemize}
\item If $b\neq  0$, we let $\mathbf w$ be the unique solution to the equation $(A+bI)\mathbf w = \mathbf c$. Then $\hat \Sigma$ satisfies an equation of the form \eqref{RotScal}, so $\Sigma$ dilates and rotates around the point $\mathbf w$.
\item If $b=0$, we let $\mathbf w$ be such that  $\mathbf c = A\mathbf w + \mathbf c_0$ where $A \mathbf c_0 = \mathbf 0$. Then $\hat \Sigma$ satisfies an equation of the form \eqref{RotTran}, so $\Sigma$ translates and rotates around the point $\mathbf w$.
\end{itemize}
In short, the general case can always be reduced to one of the two cases already covered. Hence, they give all possible self-similar motions of immersed hypersurfaces under the MCF.

\section{Helicoidal surfaces}
\label{helicoidalsurfaces}

For $h \in \mathbf R$, let $\gamma_t^h:\mathbf R^3 \rightarrow \mathbf R^3$ be the one-parameter subgroup of the group of rigid motions of $\mathbf R^3$ given by
\begin{equation*}
\gamma_t^h(x,y,z) = (x\cos t - y\sin t, x \sin t + y \cos t, z+ht), \quad t \in \mathbf R.
\end{equation*}
This motion is called a helicoidal motion with axis the $z$-axis and pitch $h$. A \emph{helicoidal surface with axis the $z$-axis and pitch $h$} is a surface that is invariant under $\gamma_t^h$ for all $t$. When $h=0$ it reduces to a rotationally symmetric surface, but in this paper we will focus on the case $h\neq0$. By reflecting across the $xy$-plane if necessary, we can assume $h>0$.
Now, the parameter $\frac 1 h$ represents the angular speed of the rotation as we move along the $z$-axis with unit speed. By setting that parameter equal to 0 (corresponding to the limit case $h\rightarrow \infty$) we get surfaces invariant under translation along the $z$-axis. Most surface equations in this paper extend smoothly to that case so we will often use the notation $h=\infty$.

A helicoidal surface with $h>0$ can be parametrized in the following way. Let $X: \mathbf R \rightarrow \mathbf R^2$ be an immersed curve in the $xy$-plane, parametrized by arc length. We translate the curve $X$ along the $z$-axis with speed $h$ and at the same time rotate it counterclockwise around with $z$-axis with unit angular speed. Then the curve traces an immersed helicoidal surface $\Sigma$ with axis the $z$-axis and pitch $h$, whose parametrization is given by
\begin{equation*}
F(s,t) = (e^{i t}X(s),ht), \quad s\in \mathbf R, t \in \mathbf R.
\end{equation*}
Note that in the literature, it is more common to parametrize the surface so that the generating curve lies in the $xz$-plane. There is also the so-called natural parametrization used in \cite{docarmdajcz}. 

The following notation will be used. Let $T = \frac{dX}{ds}$ be the unit tangent along $X$ and $N = iT$ its leftward pointing normal. The signed curvature of $X$ is $k = \langle \frac{d^2X}{ds^2},N \rangle$. It turns out to be convenient to work with the functions $\tau = \langle X, T \rangle$ and $\nu = \langle X, N \rangle$. They satisfy
\begin{equation}
\label{TauNuDiffur}
\begin{aligned}
\frac{d}{ds} \tau &= 1 + k\nu \\
\frac{d}{ds}\nu &= - k\tau
\end{aligned}
\end{equation}
and of course we have
\begin{equation}
\label{Grunnjafna}
X = (\tau + i \nu)T, \quad r^2:=|X|^2 = \tau^2+\nu^2.
\end{equation}
An interesting geometric interpretation of $(\tau,\nu)$ is given in \cite{perd}. 

Now, the tangent space of $\Sigma$ is spanned by the coordinate tangent vectors
\begin{equation*}
\frac{\partial F}{\partial s} = (e^{i t}T,0), \quad
\frac{\partial F}{\partial t} = (i e^{it}X,h), 
\end{equation*}
and one easily verifies that a unit normal is given by
\begin{equation*}
\mathbf{n}= \frac{(he^{i t}N,-\tau)}{\sqrt{\tau^2+h^2}}.
\end{equation*}
The metric is given by the matrix
\begin{equation*}
g_{ij} = 
\left(
\begin{array}{cc}
  1 &  - \nu   \\
  - \nu &   r^2+h^2  
\end{array}
\right)
\end{equation*}
and the inverse metric by
\begin{equation*}
g^{ij} = \frac{1}{\tau^2+h^2}
\left(
\begin{array}{cc}
  r^2+h^2   &  \nu   \\
  \nu &  1
\end{array}
\right).
\end{equation*}
The second derivatives are
\begin{equation*}
\frac{\partial^2F}{\partial s^2} = (e^{i t}kN,0), \quad
\frac{\partial^2F}{\partial s \partial t} = (ie^{i t}T,0), \quad
\frac{\partial^2F}{\partial t^2} = (- e^{i t}X,0),
\end{equation*}
so the second fundamental form is given by the matrix
\begin{equation*}
A_{ij} = -\langle F_{ij},\mathbf n \rangle = -\frac{h}{\sqrt{\tau^2+h^2}}
\left(
\begin{array}{cc}
  k &  1   \\
  1 &  - \nu  
\end{array}
\right).
\end{equation*}
Finally, the mean curvature of $\Sigma$ is
\begin{equation}
\label{meancurvature}
H = g^{ij}A_{ij} = -\frac{h(k(r^2+h^2)+\nu)}{(\tau^2+h^2)^{\frac 3 2}}.
\end{equation}

The following theorem shows how to get a helicoidal surface of prescribed mean curvature.

\begin{thm}
\label{prescribed}
For every smooth function $\Psi : \mathbf R^2 \rightarrow \mathbf R$ and $h>0$, there exists a complete immersed helicoidal surface of pitch $h$ satisfying the equation $H = \Psi(\tau,\nu)$.
\end{thm}

\begin{proof}
By solving Equation \eqref{meancurvature} for $k$, we have $k$ written as a smooth function of $\tau$ and $\nu$, so the proof is reduced to the following lemma.
\end{proof}

\begin{lem}
\label{FerillFenginn}
For every smooth function $\Phi:\mathbf R^2\rightarrow \mathbf R$, point $z_0\in\mathbf R^2$ and angle $\theta_0 \in [0,2\pi)$, there is a unique immersed curve $X:\mathbf R \rightarrow \mathbf R^2$ satisfying the equation $k = \Phi(\tau,\nu)$ and going through $z_0$ with angle $\theta_0$.
\end{lem}
\begin{proof}
Keeping in mind \eqref{TauNuDiffur} and \eqref{Grunnjafna}, we let $\tau, \nu, \theta$ be the unique solution to the ODE system
\begin{equation}
\label{taunuODE}
\left\{
\begin{aligned}
\tau' &= 1 + \nu\Phi(\tau,\nu) \\
\nu' &= -\tau\Phi(\tau,\nu)\\
\theta' &= \Phi(\tau,\nu)
\end{aligned} \right.
\end{equation}
with initial values $\theta(0) = \theta_0$, $\tau(0)+i\nu(0) = e^{-i\theta_0}z_0$, and then define the curve as
\begin{equation*}
 X = (\tau+i\nu)e^{i\theta}.
\end{equation*}
Since 
\begin{equation}
\label{rmat}
\frac{d}{ds}\sqrt{\tau^2+\nu^2} = \frac{\tau\tau' + \nu\nu'}{\sqrt{\tau^2+\nu^2}} = \frac{\tau}{\sqrt{\tau^2+\nu^2}} \leq 1,
\end{equation}
the solution cannot blow up in finite time and hence is defined on all of $\mathbf R$. Note that
\begin{equation*}
X' = (\tau'+i\nu'+i\theta'(\tau+i\nu))e^{i\theta} = e^{i\theta},
\end{equation*}
so $X$ is parametrized by arc length with tangent $T = e^{i\theta}$ and hence the curvature $k$ is equal to $\theta' = \Phi(\tau,\nu)$. Finally,
\begin{equation*}\begin{aligned}
\langle X,T \rangle &=\text{Re}(Xe^{-i\theta} )= \tau, \\
\langle X,N \rangle &= \text{Re}(X(-i)e^{-i\theta} )= \nu,
\end{aligned}\end{equation*}
finishing the proof.
\end{proof}

Theorem \ref{prescribed} generalizes in some sense the results in \cite{baikkouf}, which treats the case when $H$ is a function of $r = \sqrt{\tau^2+\nu^2}$. Moreover, our method always gives a complete surface. However, \cite{baikkouf} gives an explicit integral formula.

Typically the system of ODEs \eqref{taunuODE} for $\tau$ and $\nu$ has a one-parameter trajectory space. Therefore, we get a one-parameter family of surfaces for each $h$.

In the remaining three sections, we look separately at three types of helicoidal surfaces, namely surfaces rotating under the MCF and then the classical minimal and constant mean curvature helicoidal surfaces.

\section{Helicoidal surfaces rotating/translating under the MCF}
\label{rotatingsurfaces}

In this section, we describe a two-parameter family of immersed helicoidal surfaces that rotate around the $z$-axis with unit speed under the MCF. Since the surfaces are invariant under the helicoidal motion $\gamma_t^h$, this motion can also be seen as a translation with speed $h$ in the negative $z$-direction.

The matrix $R =
\left(
\begin{smallmatrix}
  0& -1  &0   \\
  1&   0&  0 \\
  0&   0&   0
\end{smallmatrix}
\right)
$ generates the rotation, so by Equation \eqref{RotTran} we get a surface rotating with unit speed if and only if
\begin{equation*}
-H = \langle RF, \mathbf{n} \rangle = \frac{h\tau}{\sqrt{\tau^2+h^2}}.
\end{equation*}
Since $H$ is given as a smooth function of $\tau$, the existence of these surfaces is guaranteed by Theorem \ref{prescribed}. The corresponding equation for the curve $X$ is
\begin{equation}
\label{SnuningsGrunnjafna}
k = \tau \frac{\tau^2+h^2}{r^2+h^2} - \frac{\nu}{r^2+h^2}
\end{equation}
so the two-dimensional system of ODEs for $\tau$ and $\nu$ becomes
\begin{equation}
\label{Snuningshneppi}
\left\{
\begin{aligned}
\tau' &= \frac{\tau^2+h^2}{r^2+h^2}(1+\tau\nu) \\
\nu'&= \frac{\tau\nu-\tau^2(\tau^2+h^2)}{r^2+h^2}.
\end{aligned}
\right.
\end{equation}

Note that the right hand side of \eqref{Snuningshneppi} remains the same when $(\tau,\nu)$ is replaced by $(-\tau,-\nu)$. Therefore, $s \mapsto -(\tau(-s),\nu(-s))$ is also a solution to the system, which of course just corresponds to the curve $X$ parametrized backwards. This symmetry will simplify some of our arguments. Also note that the system has no equilibrium points.

The limit case $h= \infty$ yields $k=\tau$, the rotation equation for the curve shortening flow in the plane, and the rotating surface is simply the product $X \times \mathbf R$. These curves were described by the author in \cite{hph}.
It turns out that when $h$ is finite, $X$ still has most of the same properties although it no longer is always embedded.

\begin{thm}
\label{Snuningssetning}
For each $h > 0$ there exists a one-parameter family of helicoidal surfaces with pitch $h$, that rotate with unit speed around their helicoidal axis under the MCF. The corresponding generating curves have one point closest to the origin and consist of two properly embedded arms coming from this point which strictly go away to infinity. Each arm has infinite total curvature and spirals infinitely many circles around the origin. The curvature goes to $0$ along each arm and the limiting angle between the tangent and the location is $\frac \pi 2$.
\end{thm}

Examples of the generating curves appear in Figures \ref{fig:rot1}-\ref{fig:rot3}.\\

The proof will be given through a series of lemmas. We start by investigating the limit behavior of $\tau$ and $\nu$.
\begin{lem}
\label{limtil}
Both $\tau$ and $\nu$ have (possibly infinite) limits as $s \rightarrow \pm \infty$.
\end{lem}

\begin{proof}

Note that if $\tau(s) = 0$, then $\tau'(s) = \frac{h^2}{\nu^2(s)+h^2} > 0$ so $\tau$ has at most one zero, and is negative before it and positive after it.

If $k(s)=0$, then it can be shown that $k'(s)>0$, by differentiating \eqref{SnuningsGrunnjafna} directly and using the fact that $\tau'(s) =1$, $\nu'(s)=0$ and $\tau(s)\nu(s) \geq 0$. Therefore,  $k$ also has at most one zero, and is positive before it and negative after it.

Since $\nu' = -k\tau$, $\nu$ has at most two extrema and therefore has a (possibly infinite) limit in each direction.

If $\tau'(s)=0$, then by differentiating $\tau' = 1+k\nu$ and using \eqref{SnuningsGrunnjafna} and the fact that $\tau(s)\nu(s) = k(s)\nu(s) = -1$, we get that $\tau''(s) = - \frac{k'(s)+\tau^4(s)}{\tau(s)}$ and $k'(s)+\tau^4(s) > 0$. Hence $\tau$ has a local minimum at $s$ if $\tau(s)<0$ and a local maximum if $\tau(s)>0$. But that means $\tau$ has at most two extrema and hence has a (possibly infinite) limit in each direction.
\end{proof}

\begin{lem}
\label{ratio}
The ratio $\frac{\nu}{\tau}$ is decreasing wherever both $\tau$ and $\nu$ are positive.
\end{lem}

\begin{proof}
Direct calculations yield
\begin{equation*}
\frac{d}{ds}\frac{\nu}{\tau} = - \frac{\tau^5+\tau^3\nu^2+h^2\tau^3+h^2\tau\nu^2+h^2\nu}{\tau^2(r^2+h^2)},
\end{equation*}
making the statement obvious.
\end{proof}

\begin{lem}
\label{nulim}
$\lim_{s\rightarrow \pm \infty}\nu = \mp \infty$.
\end{lem}
\begin{proof}
Assume $\lim_{s\rightarrow \infty}\nu$ is finite. Since there are no equilibrium points,  $\lim_{s\rightarrow\infty}\tau$ cannot also be finite so it has to be either $\infty$ or $-\infty$. But by \eqref{SnuningsGrunnjafna}, that implies $k$ goes to $\infty$ or $-\infty$ (respectively) and hence $\nu' = -k\tau$ goes to $-\infty$, a contradiction.

Assume $\lim_{s\rightarrow \infty}\nu = \infty$. If $\lim_{s\rightarrow \infty}\tau = -\infty$, then by \eqref{SnuningsGrunnjafna}, eventually $k<0$ and thus $\nu' = -k\tau <0$, a contradiction. Assume $\tau$ has a finite limit. By \eqref{SnuningsGrunnjafna}, eventually $k<0$ and since $-k\tau = \nu' >0$, $\tau$ eventually becomes positive. But Lemma \ref{ratio} then gives us an estimate $\nu \leq C \tau$, so $\tau$ is forced to go to $\infty$, a contradiction. Finally, assume $\tau$ goes to $\infty$. Then eventually both $\tau$ and $\nu$ are positive so by Lemma \ref{ratio} we get an estimate $\nu \leq C \tau$. By  \eqref{SnuningsGrunnjafna}, that implies $k\rightarrow \infty$ and thus $\nu' = -k\tau \rightarrow -\infty$, a contradiction.

The limits in the other direction follow by symmetry.
\end{proof}
\begin{lem}
\label{taulim}
$\tau$ has a finite limit in each direction.
\end{lem}
\begin{proof}
If $\lim_{s \rightarrow \infty} \tau= \infty$, then, by Lemma \ref{nulim}, eventually $1+\tau\nu <0$ so $\tau' <0$ by \eqref{Snuningshneppi}, a contradiction. Similarly, if  $\lim_{s \rightarrow \infty} \tau= -\infty$ then eventually $1+\tau\nu >0$ so $\tau' >0$ by \eqref{Snuningshneppi}, a contradiction. The limits in the other direction follow by symmetry.
\end{proof}

\begin{figure}
  \subfloat[The symmetric curve ($h=1$)]{\label{fig:roth1y0}\includegraphics{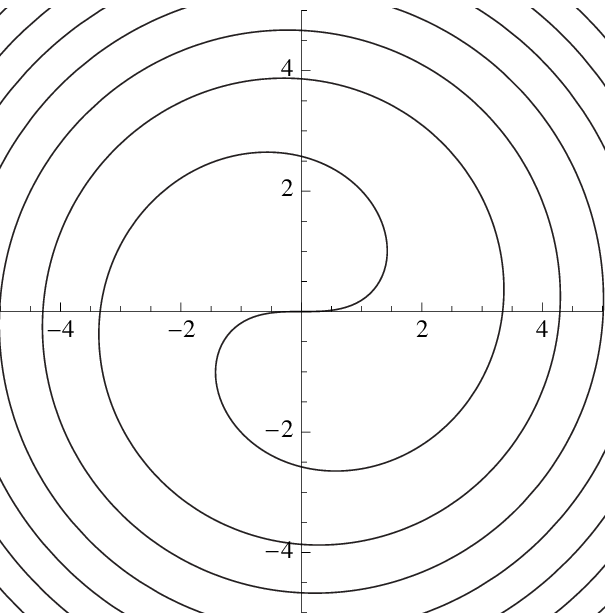}} \quad             
  \subfloat[Curve at distance 1 to the origin ($h=1$)]{\label{fig:roth1y1}\includegraphics{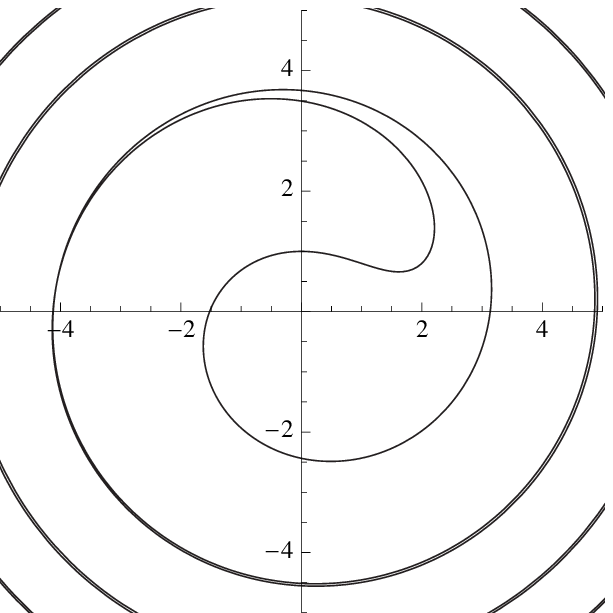}}\\
  \caption{The symmetric curve is always embedded but when the distance to the origin has reached a certain value (depending on $h$) the curve stops being embedded.}
  \label{fig:rot1}
\end{figure}

\begin{lem}
\label{klim}
$\lim_{s\rightarrow \pm \infty}k = 0\pm$.
\end{lem}

\begin{proof}
This follows directly from \eqref{SnuningsGrunnjafna} and Lemmas \ref{nulim} and \ref{taulim}. 
\end{proof}

\begin{lem}
\label{rlim}
$\lim_{s\rightarrow \pm \infty}r = \infty$ and $r$ has exactly one extremum, a global mininum.
\end{lem}
\begin{proof}
The first statement follows from Lemma \ref{nulim}, since $r^2 = \tau^2+\nu^2$. The second one follows from the fact that $\frac{d}{ds}r^2 = 2\tau$ and the observation in the proof of Lemma \ref{limtil} that $\tau$ has at most one zero.
\end{proof}

A consequence of the previous two lemmas is that each of $\tau$ and $k$ does indeed have a zero, and is negative before it and positive after it.

Lemma \ref{rlim} implies that $X$ has a unique point closest to the origin and consists of two arms coming out from this point and strictly going away to infinity. Hence, each of these arms is properly embedded but they can cross each other. The limiting growing direction of the arms is given by the following lemma.

\begin{lem}
\label{growdir}
$\lim_{s\rightarrow \pm \infty}\frac{rT}{X} = \pm i$.
\end{lem}
\begin{proof}
This follows directly from Lemmas \ref{nulim} and \ref{taulim}, since $\frac{rT}{X} = \frac{\tau-i\nu}{r}$.
\end{proof}

Now, let $\phi$ be the angle of $X$, i.e., $X = re^{i\phi}$. Then we have the following.
\begin{lem}
\label{philim}
$\lim_{s\rightarrow \pm \infty}\phi = + \infty$.
\end{lem}
\begin{proof}
Since $\frac{d\phi}{d\log(r)} = r\frac{d\phi/ds}{dr/ds}=r\frac{-\nu/r^2}{\tau/r} =-\frac{\nu}{\tau}$, which goes to $\infty$ in each direction, the result follows from Lemma \ref{rlim}.
\end{proof}

\begin{figure}
\subfloat[Curve at distance 1 to the origin ($h=5$)]{\label{fig:roth5y1}\includegraphics{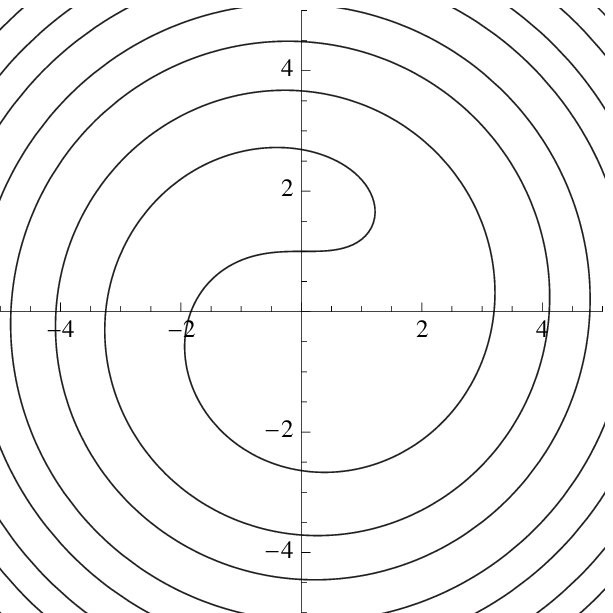}}
\quad
\subfloat[Curve passing close to the origin ($h=\frac{1}{16}$)]{\label{fig:roth00625y1x1}\includegraphics{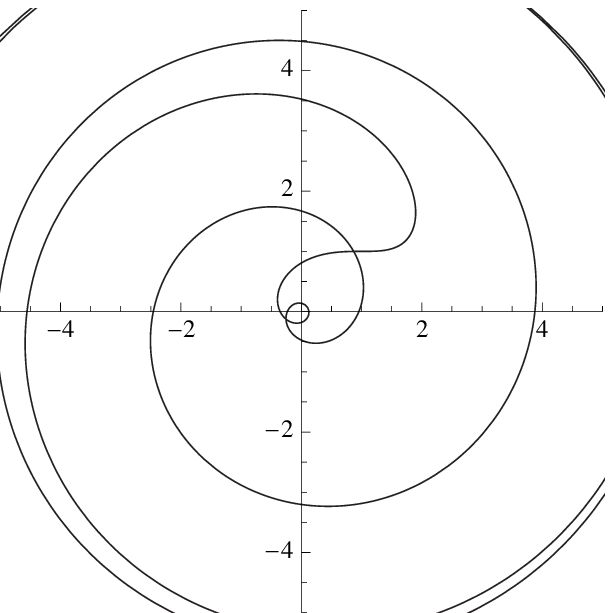}}
\caption{As $h$ goes to $\infty$, the curves converge to the rotating curves of the curve shortening flow. When $h$ is close to 0, the curvature can increase dramatically where the curve passes close to the origin.}
\label{fig:rot2}
\end{figure}

This means each arm spirals infinitely many times around the origin. Moreover, we have the following.

\begin{lem}

$\lim_{s \rightarrow \pm \infty}\theta = +\infty$.
In other words, $\int_{s_0}^\infty k ds = +\infty$ and $\int_{-\infty}^{s_0}k ds = -\infty$, so each arm has infinite total curvature.
\end{lem}
\begin{proof}
Follows from the identity $\phi = \theta + \arg(\tau+i\nu)$, Lemma \ref{philim} and the fact that $\arg(\tau+i\nu)$ has finite limit in each direction.
\end{proof}
This concludes the proof of Theorem \ref{Snuningssetning}.\\

Our parametrization for the helicoidal surfaces does not remain valid in the case $h=0$, but we can nevertheless investigate what happens to the generating curves in the limit $h \rightarrow 0$. When $h=0$, Equations \eqref{SnuningsGrunnjafna} and \eqref{Snuningshneppi} become
\begin{equation*}
k = \frac{\tau^3-\nu}{r^2} \quad\quad \text{and}\quad \quad\left\{
\begin{aligned}
\tau' &= \frac{\tau^2}{r^2}(1+\tau\nu) \\
\nu'&= \frac{\tau\nu-\tau^4}{r^2}.
\end{aligned}
\right.
\end{equation*}

In this ODE system, every point on the $\nu$-axis except for (0,0) is a fixed point. These fixed points correspond to solutions in which the curve $X$ is a circle around the origin. The other trajectories are found by noticing that the function $\frac{\nu}{\tau} + \frac 1 2 r^2$ is constant. Therefore, the trajectories are the algebraic curves $\tau^3 + \tau\nu^2+ 2\nu = 2a\tau$ where $a$ is any real constant, representing the slope of the trajectory as it goes through the origin. Unless $a=0$, the curvature $k$ blows up as the trajectory goes through the origin.
Therefore, we can only take half of the trajectory to create the smooth curve $X$. That curve $X$ is embedded, has one end spiraling out to infinity and the other one spiraling into the origin in finite time where its curvature blows up. When $a=0$, we get a complete embedded curve $X$.

These curves give us an idea about the limiting behaviour of our generating curves as $h\rightarrow 0$. In the case of the circular curves corresponding to the fixed points on the $\nu$-axis, it is easy to make a precise convergence statement.

\begin{thm}
\label{samleitni}
Assume $\Phi$ is a smooth function on $\mathbf R^3 \setminus \{(0,0,0)\}$ and $A\neq0$. For each $h\geq0$, let $(\tau_h, \nu_h)$ be the solution to the initial value problem
\begin{equation*}
\left\{
\begin{aligned}
& \tau_h' = 1 + \nu_h\Phi(\tau_h,\nu_h,h)\\
& \nu_h' = -\tau_h\Phi(\tau_h,\nu_h,h)\\
& \tau_h(0) = 0 \\
& \nu_h(0) = A.
& \end{aligned}
\right.
\end{equation*}
If $(\tau_0,\nu_0)$ is the constant solution $(0,A)$, then as $h \rightarrow 0$, $(\tau_h, \nu_h) \rightarrow (\tau_0,\nu_0)$ in the sense of uniform $C^k$-convergence on compact subsets of $\mathbf R$ for each k.
\end{thm}

\begin{proof}
Here is a sketch of the proof. Let $K$ be the annulus in the $\tau\nu$-plane defined by $\frac 1 2 |A| \leq r \leq \frac 3 2 |A|$ and let $L$ be the ball or radius $\frac 1 4 |A|$ around $(0,A)$. By using the fundamental estimate \eqref{klemma} below, we get uniform  $C^k$-convergence to the constant solution $(0,A)$ on any interval where $(\tau_h,\nu_h)\in K$. In particular, $(\tau_h,\nu_h)$ goes uniformly into $L$. Then, since $|r_h'| \leq 1$ by \eqref{rmat}, we can extend uniformly by a fixed amount the interval where $(\tau_h,\nu_h)\in K$. Repeating this process, we manage to cover all of $\mathbf R$.\\

To simplify notation, let
\begin{equation*}
F_0(\tau,\nu,h)= 1 + \nu\Phi(\tau,\nu,h), \quad G_0(\tau,\nu,h)=-\tau\Phi(\tau,\nu,h)
\end{equation*}
and for $k\geq 0$
\begin{equation*}
F_{k+1} = \frac{\partial F_k}{\partial \tau}F_0 + \frac{\partial F_k}{\partial \nu}G_0, \quad 
G_{k+1} = \frac{\partial G_k}{\partial \tau}F_0 + \frac{\partial G_k}{\partial \nu}G_0.
\end{equation*}
The functions $F_k$ and $G_k$ are smooth on $\mathbf R^3 \setminus \{(0,0,0)\}$ and
\begin{equation}
\label{AukidHneppi}
\left\{
\begin{aligned}
& \tau_h^{(k+1)} = F_k(\tau_h,\nu_h,h)\\
& \nu_h^{(k+1)} = G_k(\tau_h,\nu_h,h).
& \end{aligned}
\right.
\end{equation}
Define the constants $D_0=0$ and for $k \geq 0$,
\begin{equation*}
D_{k+1}=\sup_{K \times [0,1]}(|\nabla F_k|+|\nabla G_k|).
\end{equation*}

\begin{figure}
\subfloat[Curve at distance 1 to the origin ($h=\frac 1 8$)]{\label{fig:roth0125y1}\includegraphics{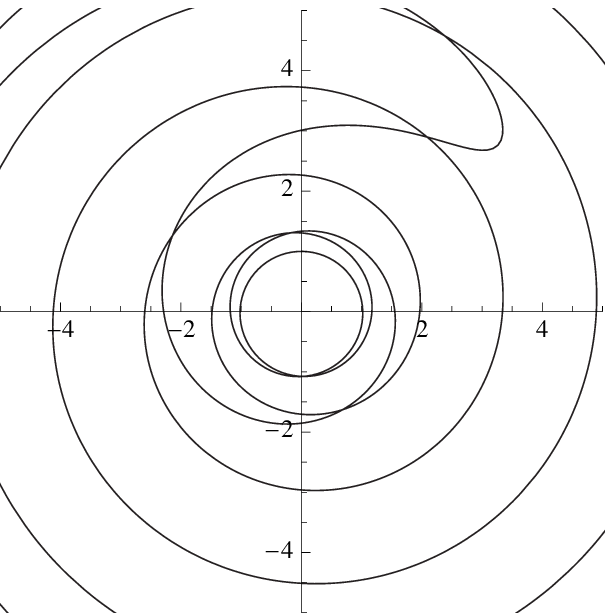}} \quad
\subfloat[Curve at distance 1 to the origin ($h=\frac{1}{16}$)]{\label{fig:roth00625y1}\includegraphics{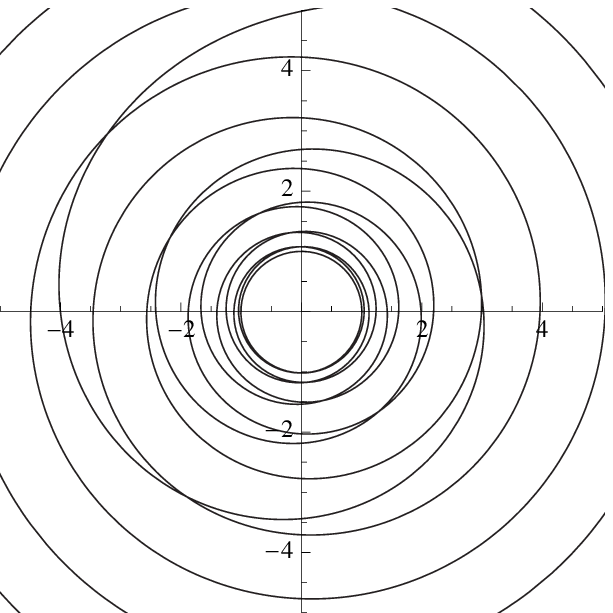}} 
\caption{As $h$ decreases to $0$, the curve at distance $1$ to the origin converges to the unit circle.}
\label{fig:rot3}
\end{figure}

Now, imagine we fix some $h\in [0,1]$ and let $I$ be an interval containing $0$ such that $(\tau_h(s),\nu_h(s)) \in K$ for all $s\in I$. Then, by the fundamental theorem of calculus and \eqref{AukidHneppi}, we have for $s \in I$
\begin{equation*}
 \begin{aligned}
 |\tau_h^{(k)}(s)-\tau_0^{(k)}(s)| &\leq  |\tau_h^{(k)}(0)-\tau_0^{(k)}(0)|+\int_0^s|\tau_h^{(k+1)}(t)-\tau_0^{(k+1)}(t)|dt \\
 &\leq hD_k + \int_0^s |F_k(\tau_h(t),\nu_h(t),h)-F_k(\tau_0(t),\nu_0(t),0)|dt\\
 &\leq hD_k + \int_0^s|F_k(\tau_h(t),\nu_h(t),h)-F_k(\tau_h(t),\nu_h(t),0)| dt \\&+ \int_0^s|F_k(\tau_h(t),\nu_h(t),0)-F_k(\tau_0(t),\nu_0(t),0)| dt.
 \end{aligned}
\end{equation*}
Here and below, the limits of integration should be swapped if $s<0$. Now, 
the first integrand is bounded by $h\sup_{K \times [0,1]}|\nabla F_k|$. The second is bounded by $|(\tau_h(t),\nu_h(t))-(\tau_0(t),\nu_0(t))| \frac \pi 2 \sup_{K \times \{0\}}|\nabla F_k|$, since every two points $p,q \in K$ can be joined by a piecewise smooth path in $K$ of length less than or equal to $\frac  \pi 2 |p-q|$. Therefore, if we put $C_k = \pi D_{k+1}$, the estimate above and the corresponding one for $\nu$ yield for $s \in I$
\begin{equation}
\label{lokamat}
\begin{aligned}
&\left|(\tau_h^{(k)}(s),\nu_h^{(k)}(s))-(\tau_0^{(k)}(s),\nu_0^{(k)}(s))\right|  \\ & \leq 2hD_k + C_k \int_0^s h +  |(\tau_h(t),\nu_h(t))-(\tau_0(t),\nu_0(t))| dt.
\end{aligned}
\end{equation}
In the case $k=0$, we let $u(s) = h + |(\tau_h(s),\nu_h(s))-(\tau_0(s),\nu_0(s))|$ and then \eqref{lokamat} becomes
\begin{equation*}
u(s) \leq h + C_0 \int_0^s u(t)dt, \quad s \in I.
\end{equation*}
By Gr\"onwall's inequality, this yields
\begin{equation*}
u(s) \leq h e^{C_0|s|}, \quad s \in I
\end{equation*}
or equivalently
\begin{equation*}
\left|(\tau_h(s),\nu_h(s))-(\tau_0(s),\nu_0(s))\right| \leq h(e^{C_0|s|}-1), \quad s \in I.
\end{equation*}
Putting this into \eqref{lokamat} yields our fundamental estimate, valid for all $k$
\begin{equation}
\label{klemma}
\left|(\tau_h^{(k)}(s),\nu_h^{(k)}(s))-(\tau_0^{(k)}(s),\nu_0^{(k)}(s))\right| \\  \leq h \left(2D_k + \tfrac{C_k}{C_0}(e^{C_0|s|}-1)\right), \, s \in I.
\end{equation}

We are now ready to finish the proof. Let $I_1$ be the interval $[-\frac 1 2 |A|, \frac 1 2 |A|]$. Since $|r_h'| \leq 1$ it is clear that $(\tau_h(s),\nu_h(s))\in K$ for all $s \in I_1$ and $h \in [0,1]$. Therefore, as $h \rightarrow 0$, we get by \eqref{klemma} that  $(\tau_h^{(k)},\nu_h^{(k)}) \rightarrow (\tau_0^{(k)},\nu_0^{(k)})$ uniformly on $I_1$ for each $k$. Pick $\delta_1$ so small that $(\tau_h(s),\nu_h(s)) \in L$ for all $s \in I_1$ and $h \in [0,\delta_1]$.

We continue by induction. Assume that for some $n\geq 1$ we have $(\tau_h(s),\nu_h(s)) \in L$ for all $s \in I_{n}$ and $h \in [0,\delta_{n}]$. Let $I_{n+1}$ be the closed interval obtained from $I_{n}$ by extending it by $\frac 1 4 |A|$ in each direction. Since $|r_h'| \leq 1$ we have $(\tau_h(s),\nu_h(s)) \in K$ for all $s \in I_{n+1}$ if $h \in [0,\delta_{n}]$. Then as $h \rightarrow 0$, by \eqref{klemma}, $(\tau_h^{(k)},\nu_h^{(k)}) \rightarrow (\tau_0^{(k)},\nu_0^{(k)})$ uniformly on $I_{n+1}$ for each $k$. Finally, we pick $\delta_{n+1}$ so small that $(\tau_h(s),\nu_h(s)) \in L$ for all $s \in I_{n+1}$ if $h \in [0,\delta_{n+1}]$.
\end{proof}

\begin{cor}
Under the assumptions of Theorem \ref{samleitni}, let $X_h$ be the curve corresponding to ($\tau_h,\nu_h)$ for some fixed choice of $\theta_0$. Then, as $h \rightarrow 0$, $X_h$ converges to the circle $s \mapsto iAe^{-i\frac s A + i\theta_0}$, $C^k$-uniformly on compact subsets of $\mathbf R$ for each $k$.
\end{cor}

\begin{proof}

Since $X_h(s) = (\tau_h(s)+i\nu_h(s))e^{i\theta_h(s)}$, where $\theta_h(s) = \int_0^sk_h(t)dt + \theta_0$ and $k_h(s) = \Phi(\tau_h(s),\nu_h(s),h)$, the result follows directly from Theorem \ref{samleitni}.
\end{proof}

For $h > 0$, circular generating curves around $(0,0)$ correspond to cylinders so this result can in some sense be interpreted as our surfaces converging to a cylinder.

\section{Helicoidal minimal surfaces}
\label{minimalsurfaces}

Wunderlich studied the helicoidal minimal surfaces in \cite{wund}. We include them here to investigate their limit behaviour as $h \rightarrow 0$.

From Equation \eqref{meancurvature} we see that the curves that generate the helicoidal minimal surfaces are given by
\begin{equation}
\label{LagmarksGrunnjafna}
k = -\frac{\nu}{r^2+h^2}.
\end{equation}
Hence, the two-dimensional system of ODEs for $\tau$ and $\nu$ becomes
\begin{equation}
\label{Lagmarkshneppi}
\left\{
\begin{aligned}
\tau' &= \frac{\tau^2+h^2}{r^2+h^2} \\
\nu'&= \frac{\tau\nu}{r^2+h^2}.
\end{aligned}
\right.
\end{equation}

The limit case $h =  \infty$ yields the equation $k=0$. That means $X$ is just a straight line and the surface $\Sigma$ becomes a plane, an embedded minimal surface.

In the general case, direct calculations yield that the function $\frac{\nu}{\sqrt{\tau^2+h^2}}$ is constant. That gives $\tau = \frac{h^2s}{A^2+h^2}$ and by solving for $\nu$, $k$ and $\theta$ in terms of $\tau$ we get that the curve is given by
\begin{equation}
\label{lagmarksferill}
X = (\tau + i \tfrac{A}{h}\sqrt{\tau^2+h^2})e^{-i\frac{A}{h} \text{arsinh}\frac{\tau}{h} +i\theta_0},\quad \tau \in \mathbf R,
\end{equation}
where $A$ in any real constant.
Letting $u = \frac 1 h \text{arsinh}\frac{\tau}{h}$ yields Wunderlich's parametrization
\begin{equation*}
X = (h\sinh hu + iA\cosh hu)e^{-i A u + i \theta_0},\quad u \in \mathbf R.
\end{equation*}
These curves appear in Figure \ref{fig:min1}.

Note that $|A|$ gives the distance of the curve from the origin. When $A=0$, $X$ is just a straight line through the origin so the surface $\Sigma$ is the helicoid, the well-known embedded minimal surface (see \cite{coldmin}). In the other cases, the curve spirals infinitely many times around the origin, intersecting itself infinitely many times. It is symmetric with respect to reflections across the line through the origin and the point $ie^{i\theta_0}$. The curvature never changes sign, reaches its maximum absolute value at the point closest to the origin and goes to zero in each direction. The limiting growing angle is given by
\begin{equation*}
\lim_{s \rightarrow \pm \infty}\frac{rT}{X}  = \frac{\pm h -iA}{\sqrt{h^2+A^2}}.
\end{equation*}

\begin{figure}
\subfloat[$h=\frac 1 2$, $A=1$]{\label{fig:minh05y1}\includegraphics{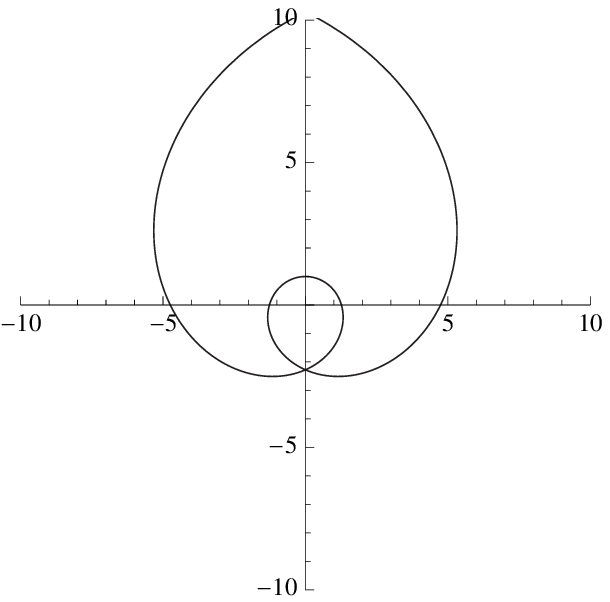}} \quad
\subfloat[$h=\frac{1}{8}$, $A=1$]{\label{fig:minh0125y1}\includegraphics{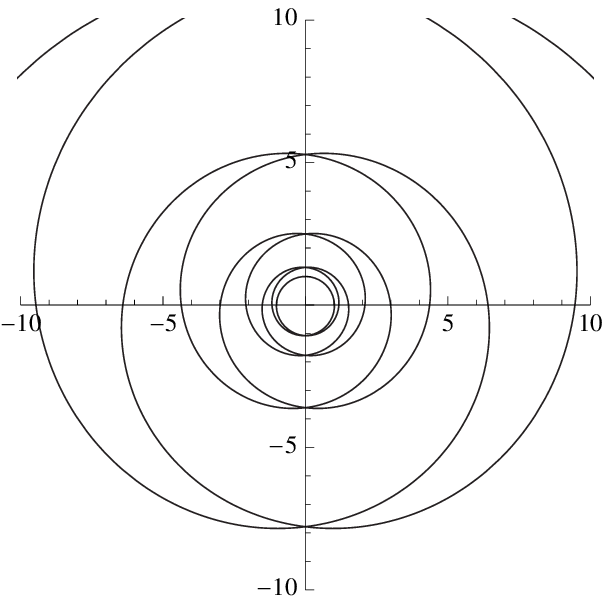}} 
\caption{As $h$ decreases to $0$, the curve at distance $1$ to the origin converges to the unit circle.}
\label{fig:min1}
\end{figure}

In the limit $h \rightarrow 0$, Equations \eqref{LagmarksGrunnjafna} and \eqref{Lagmarkshneppi} become
\begin{equation*}
k = -\frac{\nu}{r^2} \quad\quad \text{and}\quad \quad\left\{
\begin{aligned}
\tau' &= \frac{\tau^2}{r^2} \\
\nu'&= \frac{\tau\nu}{r^2}.
\end{aligned}
\right.
\end{equation*}

Like in the rotating surface case, every nonzero point on the $\nu$-axis is a fixed point of this system. The other trajectories are found by noticing that the function $\frac{\nu}{\tau}$ is constant so they are straight lines of the form $\nu = a\tau$ for some real constant $a$. This gives $\tau' = \frac{1}{a^2+1}$ so $\tau = \frac{1}{a^2+1}s$ and thus $\nu = \frac{a}{a^2+1}s$. Moreover, $k = -\frac{a}{s}$ so $\theta = -a\log |s|$. Unless $a=0$, the curvature $k$ blows up at the origin so the corresponding curve $X$ is only defined for $s>0$ (or $s<0$) where it is given by the simple formula
\begin{equation}
\label{lagmarksmarkgildi}
X = \frac{1+ia}{a^2+1}se^{-ia\log s} = \frac{s^{1-ia}}{1-ia}, \quad s >0.
\end{equation}
These curves of course have a constant growing angle $\frac{rT}{X} = \frac{1-ia}{\sqrt{1+a^2}}$.

Just as in the rotating surface case, we can apply Theorem \ref{samleitni} to get the same type of convergence of certain generating curves to the circular curves corresponding to the fixed points on the $\nu$-axis. However, here this convergence can also be seen directly from \eqref{lagmarksferill}.

The curve given by \eqref{lagmarksmarkgildi} can also be realized as a limit. Put $A = ah$ and $\theta_0 = a \log \frac{2}{h(a^2+1)}$ into \eqref{lagmarksferill} to get
\begin{equation*}
X = \frac{s+ia\sqrt{s^2+h^2(a^2+1)^2}}{a^2+1}e^{-ia\log\frac{s+\sqrt{s^2+h^2(a^2+1)^2}}{2}}, \quad s \in \mathbf R.
\end{equation*}
When $h \rightarrow 0$, this curve converges to the curve given by \eqref{lagmarksmarkgildi}, $C^k$-uniformly on compact subets of $(0,\infty)$ for each $k$.

\section{Helicoidal surfaces of constant mean curvature}
\label{cmcsurfaces}

The constant mean curvature helicoidal surfaces  were studied by Do Carmo and Dajczer in \cite{docarmdajcz} and further by Hitt and Roussos in \cite{hittrouss} and Perdomo in \cite{perd}. We include them here to give a classification of the immersed cylinders in this family of surfaces.

From Equation \eqref{meancurvature} we see that the helicoidal surfaces of constant mean curvature $H\neq0$ are given by the equation
\begin{equation*}
k = -\frac{H(\tau^2+h^2)^{\frac 3 2}}{h(r^2+h^2)}-\frac{\nu}{r^2+h^2}.
\end{equation*}
In the special case $h = \infty$, the equation reduces to $k = -H$. Therefore, $X$ is a circle and the surface $\Sigma$ is a cylinder, a well-known embedded constant mean curvature surface.

Now, by rescaling the surface and parametrizing $X$ backwards if necessary, we can assume $H=-1$.  Then we have the following two-dimensional system of ODEs for $\tau$ and $\nu$, which  also appeared in \cite{perd}
\begin{equation}
\label{cmchneppi}
\left\{\begin{aligned}
\tau' &= \frac{\tau^2+h^2}{h(r^2+h^2)}(h+\nu\sqrt{\tau^2+h^2})\\
\nu' &= \frac{h\tau\nu-\tau(\tau^2+h^2)^{\frac 3 2}}{h(r^2+h^2)}
\end{aligned}\right.
\end{equation}
The system has a unique fixed point $(0,-1)$, corresponding to $X$ being the unit circle traced counterclockwise. To find the other trajectories we make a change of variables. First,  define the norm-preserving involution
\begin{equation*}
\Phi_{h}:\mathbf R^2 \rightarrow \mathbf R ^2, \quad (x_1,x_2) \mapsto \left(\frac{\sqrt{r^2+h^2}}{\sqrt{x_2^2+h^2}}x_2,\frac{h}{\sqrt{x_2^2+h^2}}x_1\right),
\end{equation*}
where $r^2 = x_1^2+x_2^2$. Note that $\Phi_\infty$ is just reflection across the line $x_1=x_2$. Then let $x$ and $y$ be the functions given by $(x,y) = \Phi_h(\nu,\tau)$. Direct calculations using \eqref{cmchneppi} yield that the functions $x$ and $y$ are solutions to the system of ODEs
\begin{equation*}
\left\{
\begin{aligned}
x' &= \frac{h}{\sqrt{y^2+h^2}}(y+1) \\
y' &= \frac{-hx}{\sqrt{y^2+h^2}}.
\end{aligned}\right.
\end{equation*}
Now, introduce a new variable $u$ such that $\frac{du}{ds} = \frac{h}{\sqrt{y^2+h^2}}$. Then we get the simpler system
\begin{equation*}
\left\{\begin{aligned}
\frac{dx}{du} &= y+1\\
\frac{dy}{du} &= -x,
\end{aligned}\right.
\end{equation*}
whose trajectory space consists of concentric circles around the fixed point $(0,-1)$. The solutions are parametrized by  $(x,y) = (R\cos u, -1 - R\sin u)$, where $R \in [0,\infty)$ is any constant. Now, the solutions to the original system are given by $(\nu,\tau) = \Phi_h(x,y)$, so they are also periodic. Since $\Phi_h$ is norm-preserving, each solution is contained in an annulus with center at the origin, inner radius $|R-1|$ and outer radius $R+1$. It touches each boundary once in every period. The corresponding curve $X$ is therefore also contained in the same annulus and consists of repeated identical excursions between the two boundaries of the annulus. Note that $X$ passes through the origin if and only if $R=1$.

Now, let's examine a whole excursion, i.e., we go from the outer boundary, touch the inner boundary and go back out to the outer boundary. This corresponds to the $(\tau,\nu)$ trajectory going through one period. Let $\Delta \theta$ denote the difference between the values of the tangent angle before and after the excursion. There are two possibilities:\\

1) The angle difference $\Delta \theta$ is of the form $\frac{p2\pi}{q}$ for some relatively prime, positive integers $p$ and $q$. Then after $q$ excursions the curve $X$ closes up. Thus, $X$ is a closed curve, an immersed $\mathbf S^1$, with $q$-fold rotational symmetry and the surface is an immersed cylinder. The number $p$ is the rotation number of $X$.\\

2) The angle difference $\Delta \theta$ is not of this form. In this cases the curve $X$ never closes up and takes infinitely many excursions. It is dense in the annulus.\\

In the rest of this section, we find all possible values of $\Delta \theta$ in order to obtain a classification of the closed immersed curves mentioned above, in the spirit of the classification of the immersed self-shrinkers of the curve shortening flow in the plane given by Abresch and Langer in \cite{abrlang}.\\

Since $d\theta = kds$, $\Delta \theta$ is given by integrating $k$ over one period:
\begin{equation}
\label{deltatheta}
\begin{aligned}
\Delta \theta &= \int kds = \int_{-\pi}^{\pi}\left(\frac{h\sqrt{r^2+h^2}}{y^2+h^2} - \frac{y}{h\sqrt{r^2+h^2}}\right)du \\
& = \int_{-\pi}^\pi \left(\frac{h\sqrt{R^2+1+2R\sin u+h^2}}{(1+R\sin u)^2+h^2} +\frac{1+R\sin u}{h\sqrt{R^2+1+2R\sin u+h^2}}\right)du.
\end{aligned}
\end{equation}

Clearly, $\Delta \theta$ is smooth in $R$ and $h$ on $[0,\infty)\times(0,\infty]$. Note that when $h=\infty$, $k=1$ so $\Delta \theta = 2 \pi$ for all $R$ and therefore $\frac{\partial}{\partial R}\Delta \theta = 0$.

It turns out to be convenient to work with $\Delta \phi$, where $\phi$ is the angle of the curve $X$, i.e., $X = re^{i\phi}$. Since $X = (\tau + i\nu)e^{i\theta}$, we get the following relation between $\Delta\phi$ and $\Delta\theta$, by noticing that the $(\tau,\nu)$ trajectory winds clockwise around the origin when $R>1$, goes through the origin when $R=1$ and does neither when $0 \leq R <1$:
\begin{equation}
\label{thetaphi}
\Delta \phi = \begin{cases} \Delta \theta & \text{if }0\leq R < 1 ,\\
\Delta \theta - \pi & \text{if } R=1,\\
\Delta \theta- 2\pi & \text{if } R>1.
 \end{cases} 
\end{equation}
Since $d\phi = - \frac{\nu}{r^2}ds$, $\Delta \phi$ is given by integrating over one period:
\begin{equation}
\label{angledifference}
\begin{aligned}
\Delta\phi & = -\int \frac{\nu}{r^2}ds = - \int_{\pi}^{\pi}\frac{\sqrt{r^2+h^2}}{hr^2}ydu \\
& = \int_{-\pi}^{\pi}\frac{\sqrt{R^2+1+2R\sin u + h^2}}{h(R^2+1+2R\sin u)}(1+R\sin u)du.
\end{aligned}
\end{equation}
We should mention that the integral \eqref{deltatheta} also appears in \cite{perd}, whereas the integral \eqref{angledifference} is used in \cite{hittrouss}.

\begin{lem}
When R=0, $\Delta\phi$ takes the value $\frac{2\pi\sqrt{h^2+1}}{h}$.
\end{lem}
\begin{proof}
This follows directly from setting $R=0$ in Equation \eqref{angledifference}.
\end{proof}

The complete elliptic integral of the second kind shows up in our next lemma. It is defined as
\begin{equation*}
E(k) = \int_{0}^{\frac \pi 2}\sqrt{1-k^2\sin^2\theta}d\theta.
\end{equation*}
and gives $\frac 1 4$ of the circumference of an ellipse with semi-major axis $1$ and eccentricity $k$.

\begin{lem}
When R=1, $\Delta\phi$ takes the value $\frac{2\sqrt{h^2+4}}{h}E(\frac{2}{\sqrt{h^2+4}})$.
\end{lem}
\begin{proof}
Set $R=1$ in Equation \eqref{angledifference} to get
\begin{equation*}
\begin{aligned}
\Delta \phi & = \int_{-\pi}^{\pi}\frac{\sqrt{h^2+2+2\sin u}}{2h}du \\
& = \frac{\sqrt{h^2+4}}{h}\int_{-\frac \pi 2}^{\frac \pi 2}\sqrt{1-\frac{4}{h^2+4}\sin^2\theta}d\theta,
\end{aligned}
\end{equation*} 
where in the second equation we replaced $\sin u$ with $\cos u$, set $u = 2\theta$ and used $\cos 2\theta = 1-2\sin^2\theta$.
\end{proof}
The value in the lemma is half the circumference of an ellipse with semi-minor axis $1$ and semi-major axis $\frac{\sqrt{h^2+4}}{h}$.

\begin{lem}
For each $h$, $\lim_{R \rightarrow \infty}\Delta \phi = 0$.
\end{lem}
\begin{proof}
If $R\geq 2$, then $r \geq 1$ so the integrand in \eqref{angledifference} is bounded by $\frac{\sqrt{h^2+1}}{h}$. By dominated convergence, $\lim_{R\rightarrow \infty}\Delta \phi = \int_{-\pi}^{\pi}\frac{\sin u}{h} = 0$.
\end{proof}

\begin{lem}
For each $h$, $\Delta \phi$ is strictly decreasing in $R$.
\end{lem}
\begin{proof}
We need to show that $\frac{\partial}{\partial R} \Delta \phi < 0$ when $R\neq 1$. Since $\frac{\partial}{\partial R} \Delta \phi = 0$ in the limit case $h=\infty$, it is enough to show that $\frac{\partial}{\partial h}\frac{\partial}{\partial R} \Delta \phi > 0$. Now,
\begin{equation*}
\begin{aligned}
\frac{\partial}{\partial h}\frac{\partial}{\partial R} \Delta \phi &= \frac{\partial}{\partial R}\frac{\partial}{\partial h} \Delta \phi \\
&= \frac{\partial}{\partial R}\int_{-\pi}^\pi\frac{y}{h^2\sqrt{r^2+h^2}}du\\
&= \int_{-\pi}^\pi\frac{R\cos^2 u-h^2\sin u}{h^2(R^2+1+2R\sin u + h^2)^{\frac 3 2}}du.
\end{aligned}
\end{equation*}
The first term is clearly positive. However, so is the integral of the second term, since the denominator is smaller when $\sin u$ takes its negative values.
\end{proof}
Now, from the previous lemmas and Equation \eqref{thetaphi} we conclude that for each $h$, $\Delta \theta$ is a smooth, decreasing function of $R$ satisfying
\begin{itemize}
\item $\Delta \theta = \frac{2\pi\sqrt{h^2+1}}{h}$, when $R=0$,
\item $\Delta \theta = \frac{2\sqrt{h^2+4}}{h}E(\frac{2}{\sqrt{h^2+4}})+\pi$, when $R=1$,
\item $\lim_{R\rightarrow \infty}\Delta \theta = 2\pi$.
\end{itemize}
Therefore, we have the following theorem. The trichotomy corresponds to $R=1$, $R<1$ and $R>1$ respectively, and
\begin{equation*}
\alpha_h =  \frac{\sqrt{h^2+4}}{h\pi}E(\frac{2}{\sqrt{h^2+4}})+\frac 1 2.
\end{equation*}
Note that the winding number of $X$ is given by $\frac{q\Delta \phi}{2\pi}$, and since $\Delta \theta$ is positive, our choice of unit normal was the inward pointing one. Hence $H=1$ with respect to the outward pointing unit normal.
\begin{thm}
\label{immersedcylinder}
 Let $p$ and $q$ be relatively prime positive integers, such that $1 < \frac p q < \frac{\sqrt{h^2+1}}{h}$. Then there exists a unique (up to rotation) closed immersed curve $X$, with rotation number $p$ and $q$-fold rotational symmetry, such that it generates a helicoidal surface with pitch $h$ and constant mean curvature $H=1$ w.r.t.~ the outward pointing unit normal. Moreover,
 \begin{itemize}
 \item if $\frac p q =\alpha_h$, X goes through the origin,
 \item if  $\frac p q >\alpha_h$, X has winding number $p$ around the origin,
 \item if $\frac p q < \alpha_h$, X has winding number $p-q$ around the origin.
 \end{itemize}
\end{thm}
A few of these generating curves appear in Figures \ref{fig:cmc1} - \ref{fig:cmclast}, where they are  rescaled to have the same outer radius.

\newpage

\begin{figure}
\subfloat[$h=1$, $p=4$, $q=3$]{\includegraphics{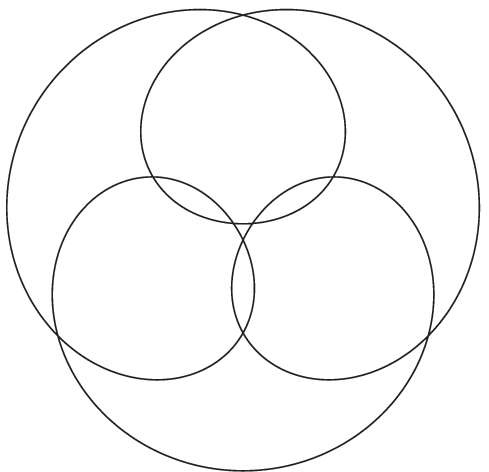}} \quad\quad\quad
\subfloat[$h=1$, $p=6$, $q=5$]{\includegraphics{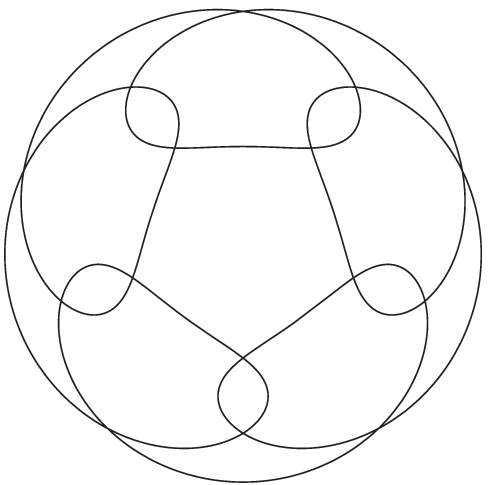}}\\
\caption{}
\label{fig:cmc1}
\end{figure}
\begin{figure}
\subfloat[$h=1$, $p=11$, $q=8$]{\includegraphics{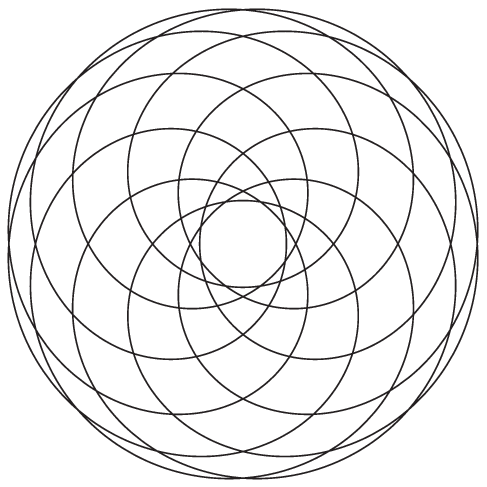}} \quad\quad\quad
\subfloat[$h=1$, $p=19$, $q=15$]{\includegraphics{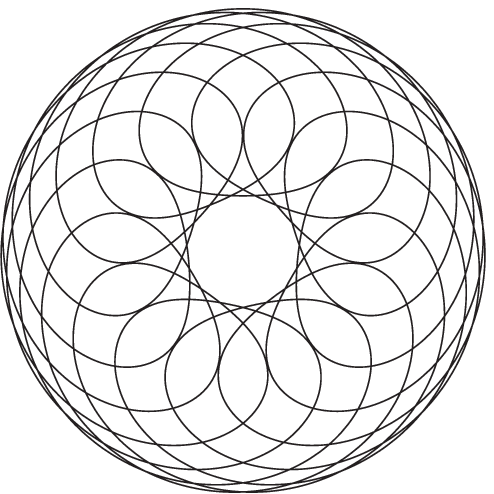}}\\
\caption{}
\end{figure}
\begin{figure}
\subfloat[$h=\frac 1 2$, $p=2$, $q=1$]{\includegraphics{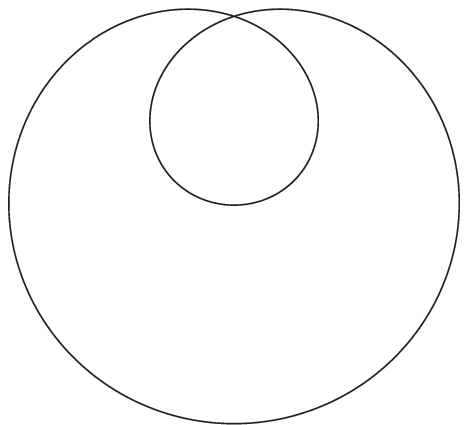}}\quad\quad\quad
\subfloat[$h=\frac 1 2$, $p=3$, $q=2$]{\includegraphics{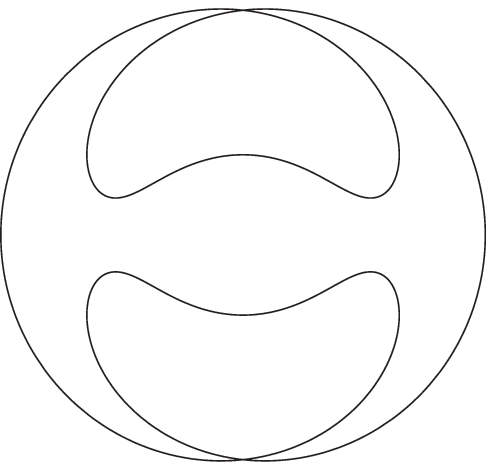}}\\
\caption{}
\end{figure}
\clearpage
\begin{figure}
 \subfloat[$h=\frac 1 2$, $p=5$, $q=3$]{\includegraphics{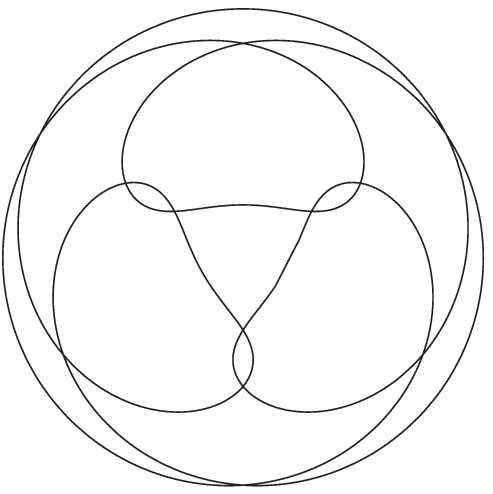}}\quad\quad\quad
\subfloat[$h=\frac 1 2$, $p=11$, $q=6$]{\includegraphics{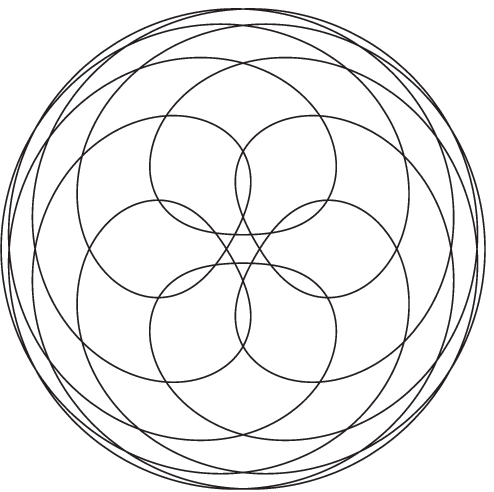}}\\
\caption{}
\end{figure}
\begin{figure}
\subfloat[$h=\frac 1 5$, $p=15$, $q=4$]{\includegraphics{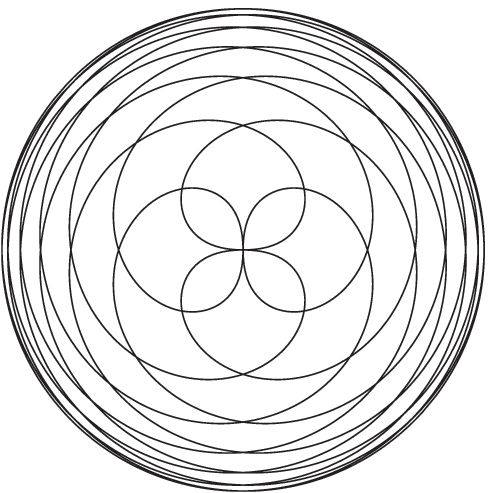}} \quad\quad\quad
\subfloat[$h=\frac 1 5$, $p=13$, $q=5$]{\includegraphics{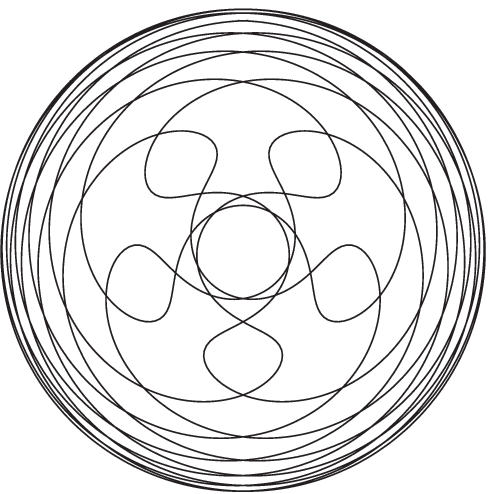}}\\
\caption{}
\end{figure}
\begin{figure}
\subfloat[$h=2$, $p=11$, $q=10$]{\includegraphics{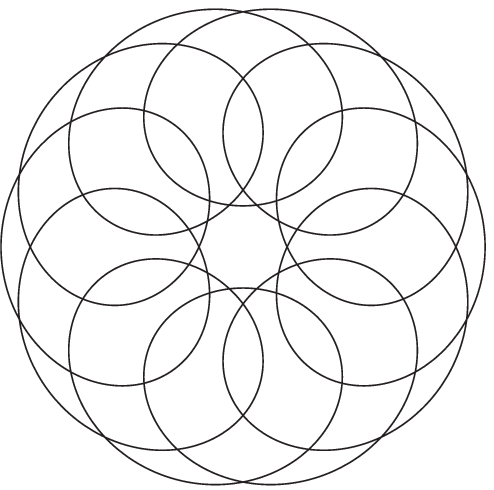}} \quad\quad\quad
\subfloat[$h=2$, $p=29$, $q=26$]{\includegraphics{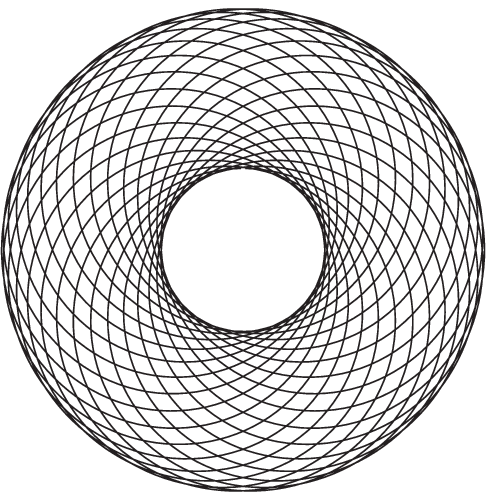}}\\
\caption{}
\end{figure}
\clearpage
\begin{figure}
\subfloat[$h=5$, $p=52$, $q=51$]{\includegraphics{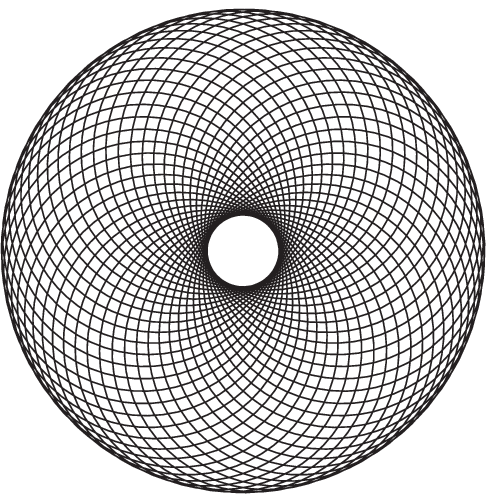}} \quad\quad\quad
\subfloat[$h=5$, $p=55$, $q=54$]{\includegraphics{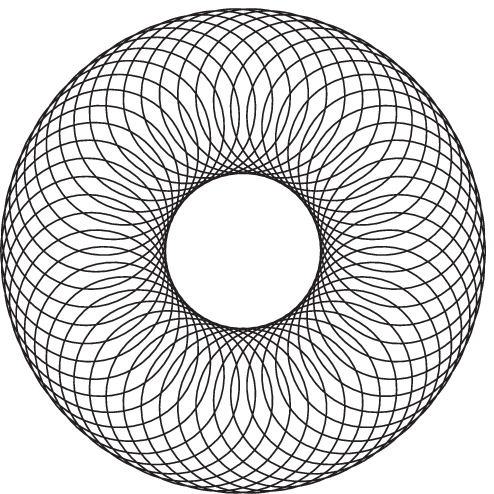}}\\
\caption{}
\label{fig:cmclast}
\end{figure}
\newpage

\section*{Acknowledgements}

The author would like to thank Eric Marberg for reading over the draft and providing valuable feedback. He would also like to thank Oscar Perdomo for sharing with him his recent paper on the CMC helicoidal surfaces. Finally, he would like to thank his advisor, Tobias Colding, for guidance and support.


\providecommand{\bysame}{\leavevmode\hbox to3em{\hrulefill}\thinspace}
\providecommand{\MR}{\relax\ifhmode\unskip\space\fi MR }
\providecommand{\MRhref}[2]{%
  \href{http://www.ams.org/mathscinet-getitem?mr=#1}{#2}
}
\providecommand{\href}[2]{#2}

\end{document}